\documentclass{amsart}
\usepackage{amsfonts,amssymb, amscd}

\newtheorem{theorem}{Theorem}[section]

\newtheorem{lemma}[theorem]{Lemma}
\newtheorem{corollary}[theorem]{Corollary}
\newtheorem{proposition}[theorem]{Proposition}
\theoremstyle{remark}

\theoremstyle{definition}
\newtheorem{definition}[theorem]{Definition}

\numberwithin{equation}{section}
\makeatother

\begin{document}

\title[Quantum measurable cardinals]{Quantum measurable cardinals}

\date{\today}

\author{David P. Blecher}
\address{Department of Mathematics, University of Houston, Houston, TX
77204-3008}
\email[David P. Blecher]{dblecher@math.uh.edu}
\author{Nik Weaver} 
\address{Department of Mathematics, Washington University, Saint Louis, MO
63130}
\email[Nik Weaver]{nweaver@math.wustl.edu}


\begin{abstract}
We investigate states on von Neumann algebras which are not normal but enjoy
various forms of infinite additivity, and show that these exist on $B(H)$ if
and only if the cardinality of an orthonormal basis of $H$ satisfies various
large cardinal conditions. For instance, there is a singular countably
additive pure state on $B(l^2(\kappa))$ if and only if $\kappa$ is Ulam
measurable, and there is a singular ${<}\,\kappa$-additive pure state on
$B(l^2(\kappa))$ if and only if $\kappa$ is measurable. The proofs make use
of Farah and Weaver's theory of quantum filters \cite{FW}. We can generalize
some of these characterizations to arbitrary von Neumann algebras.
Applications to Ueda's peak set theorem for von Neumann algebras are discussed
in the final section.
\end{abstract}

\maketitle

\vspace{.25in}

\section{Measurable cardinals}

In a recent paper \cite{blueda} Blecher and Labuschagne investigated
whether every von Neumann algebra verifies Ueda's peak set theorem
\cite{Ueda}. It was discovered that the answer turns on the existence
of singular states with a certain continuity property. Such states, which
falsify the von Neumann algebra version of Ueda's theorem, exist on
$l^\infty(\kappa)$ and $B(l^2(\kappa))$ provided $\kappa$ has a large
cardinal property related to measurability. (Throughout this paper
$\kappa$ denotes an infinite cardinal.) Similar conditions on states
were studied in the context of axiomatic von Neumann algebra quantum
mechanics in e.g.\ \cite{Ham0, BH}. This motivated us to consider the
general question of the existence of singular states on $B(H)$ with
various continuity properties. We will return to Ueda's theorem in the
final section of the paper.

In set theory there is an elaborate hierarchy of ``large cardinal'' properties,
some of which involve various notions of measurability \cite{Drake, Jech}.
Four of these are of primary interest to us here. An uncountable cardinal
$\kappa$ is said to be
\begin{itemize}
\item {\em measurable} if there is a nonzero ${<}\,\kappa$-additive
$\{0,1\}$-valued measure on $\kappa$ which vanishes on singletons

\item {\em real-valued measurable} if there is a ${<}\,\kappa$-additive
probability measure on $\kappa$ which vanishes on singletons

\item {\em Ulam measurable} if there is a nonzero countably additive
$\{0,1\}$-valued measure on $\kappa$ which vanishes on singletons

\item {\em Ulam real-valued measurable} if there is a countably additive
probability measure on $\kappa$ which vanishes on singletons.
\end{itemize}
Here a ``measure on $\kappa$'' is understood to be defined on all subsets
of $\kappa$, and ``${<}\,\kappa$-additive'' means ``additive on any family
of fewer than $\kappa$ disjoint sets''. In a set theory context the latter
property would normally be called ``$\kappa$-additive'', but that becomes
awkward in the case of countable additivity, so we shall use this slightly
modified terminology.

No cardinal of any of these types can be proven to exist in ordinary set
theory, assuming ordinary set theory is consistent. They are ``large'' in
the sense that the
smallest real-valued measurable or Ulam real-valued measurable cardinal, if
one exists, must be weakly inaccessible, and the smallest measurable or
Ulam measurable cardinal, if one exists, must be strongly inaccessible.
It is generally believed that the existence of such cardinals is
consistent with ordinary set theory. However, this, if true, could not
be proven within ordinary set theory. What we can say is that they all
have the same consistency strength, i.e., the consistency of any of the
theories
\begin{itemize}
\item[] ZFC + ``a measurable cardinal exists''
\item[] ZFC + ``a real-valued measurable cardinal exists''
\item[] ZFC + ``an Ulam measurable cardinal exists''
\item[] ZFC + ``an Ulam real-valued measurable cardinal exists''
\end{itemize}
implies the consistency of each of the others.

We summarize the basic relations among these notions. Every measurable
cardinal is real-valued measurable, and among cardinals $> 2^{\aleph_0}$
the two notions coincide. However, no measurable cardinal can be
$\leq 2^{\aleph_0}$, whereas if it is consistent that a measurable cardinal
exists then it is consistent that there is a real-valued measurable
cardinal $\leq 2^{\aleph_0}$. A cardinal is Ulam measurable if and only
if it is greater than or equal to some measurable cardinal, and a
cardinal is Ulam real-valued measurable if and only if it is greater
than or equal to some real-valued measurable cardinal. 

\section{Quantum measurability conditions}

C*-algebra and related subjects are often viewed as ``noncommutative'' or
``quantum'' analogs of classical mathematical topics, and there is a fairly
extensive body of knowledge regarding the problem of translating classical
concepts into the quantum realm. The set-theoretic aspect of this analogy
has attracted recent attention \cite{Weav1, FW}. One of the main points of
this paper is that there are good, natural quantum versions of the various
notions of measurability discussed in the last section.

It is a standard fact that a state $\phi$ on a von Neumann algebra is
normal (that is, weak* continuous) if and only if it is completely additive,
i.e., any family $\{p_\alpha\}$ of mutually orthogonal projections satisfies
$\phi(\sum p_\alpha) = \sum \phi(p_\alpha)$ \cite[Corollary 3.11]{Tak}.
Of course every state is automatically finitely additive on projections
in the sense that the preceding holds for any finite family. There is a
natural range of intermediate conditions which interpolate between these
two extremes.

\begin{definition}\label{kappaadd}
A state $\phi$ on a von Neumann algebra $M$ is {\em countably additive} if 
$\phi(\sum p_n) = \sum \phi(p_n)$ for any countable family
$\{p_n\}$ of mutually orthogonal projections.

More generally, for any infinite cardinal $\kappa$, a state $\phi$ is
{\em ${<}\,\kappa$-additive} if for every $\lambda < \kappa$ and every family
$\{p_\alpha: \alpha \in \lambda\}$ of $\lambda$ many mutually orthogonal
projections in $M$, we have
$\phi(\sum p_\alpha) = \sum \phi(p_\alpha)$.
\end{definition}

Any cardinal is the set of all smaller ordinals, so the condition
``$\alpha \in \lambda$'' means that $\alpha$ ranges over all ordinals of
cardinality smaller than $\lambda$. Again, we say ``${<}\,\kappa$-additive''
instead of ``$\kappa$-additive'' because it meshes better with the
term ``countably additive''.

This concept has a simple alternative formulation. Say that a von Neumann
algebra $M$ is {\em ${<}\,\kappa$-decomposable} if every family of mutually
orthogonal nonzero projections has cardinality strictly less than $\kappa$.

\begin{proposition}\label{addvsnormal}
Let $\phi$ be a state on a von Neumann algebra $M$ and let 
$\kappa$ be an infinite cardinal. The following are equivalent:
\begin{itemize}
\item[(i)] $\phi$ is ${<}\,\kappa$-additive
\item[(ii)] the restriction of $\phi$ to any ${<}\,\kappa$-decomposable von
Neumann subalgebra of $M$ is normal
\item[(iii)] the restriction of $\phi$ to any atomic abelian
${<}\,\kappa$-decomposable von Neumann subalgebra of $M$ is normal.
\end{itemize}
\end{proposition}

\begin{proof}
(i) $\Rightarrow$ (ii) \ If $\phi$ is ${<}\,\kappa$-additive, then its
restriction to any $\kappa$-decomposable von Neumann subalgebra of $M$
is completely additive and therefore normal.

(ii) $\Rightarrow$ (iii) \ Vacuous.

(iii) $\Rightarrow$ (i) \  Assume the restriction of $\phi$ to any atomic
abelian $\kappa$-decomposable von Neumann subalgebra of $M$ is normal and let
$\{p_\alpha: \alpha \in \lambda\}$ be a family of mutually orthogonal nonzero
projections, where $\lambda < \kappa$. Then the von Neumann subalgebra
the $p_\alpha$ generate is isomorphic to $l^\infty(\lambda)$, which is atomic,
abelian, and ${<}\, \kappa$-decomposable, so by hypothesis $\phi$ is normal
on it. Thus $\phi(\sum p_\alpha) = \sum \phi(p_\alpha)$.
\end{proof}

There are natural and easy translations of the various notions of
measurability under discussion into the language of states on atomic
abelian von Neumann algebras, in which the notion of a
${<}\,\kappa$-additive state plays a key role. A state on
$l^\infty(\kappa)$ is {\em singular} if it vanishes on all finitely
supported functions. In the next result we use integrals over finitely
additive measures; for a full treatment of this topic see \cite{Rao}.
However, for our purposes all that is needed is the observation that
if $\mu$ is a finitely additive probability measure on $\kappa$, then
there is an obvious way to define the integral of any simple function against
$\mu$, and this yields a positive norm one linear functional on the set of
all simple functions. As the simple functions are norm dense in
$l^\infty(\kappa)$, this extends to a state on $l^\infty(\kappa)$.

\begin{proposition}\label{measvsstate}
An uncountable cardinal $\kappa$ is
\begin{itemize}

\item[(i)] Ulam real-valued measurable if and only if there
is a singular countably additive state on $l^\infty(\kappa)$

\item[(ii)] Ulam measurable if and only if there is a
singular countably additive pure state on $l^\infty(\kappa)$

\item[(iii)] real-valued measurable if and only if there
is a singular ${<}\,\kappa$-additive state on $l^\infty(\kappa)$

\item[(iv)] measurable if and only if there is a singular
${<}\,\kappa$-additive pure state on $l^\infty(\kappa)$.
\end{itemize}
\end{proposition}

\begin{proof}
(i) \ There is a well-known bijective correspondence between states on
$l^\infty(\kappa)$ and finitely additive probability measures on $\kappa$.
Namely, if $\mu$ is a finitely additive probability measure on $\kappa$
then integration against $\mu$ yields a state on $l^\infty(\kappa)$ in the
manner described above, and if $\phi$ is a state on $l^\infty(\kappa)$ then
$\mu(S) = \phi(\chi_S)$ is a finitely additive probability measure. The
two constructions are inverse to each other and both preserve countable
additivity, i.e., a countably additive state becomes a countably additive
measure and vice versa. (Recall that projections in $l^\infty(\kappa)$
correspond to subsets of $\kappa$.) A finitely additive measure vanishes
on singletons (or equivalently, on finite subsets) if and only if the
corresponding state vanishes on finitely supported functions.

(iii) \ The proof is the same as the proof of (i), replacing the observation
that the two constructions preserve countable additivity with the observation
that they preserve ${<}\,\kappa$-additivity.

(ii) and (iv) \ These parts follow from (i) and (iii) plus the easy fact that a
state on $l^\infty(\kappa)$ is pure if and only if its value on any projection
is either 0 or 1.   
\end{proof}

A state on $B(l^2(\kappa))$ is {\em singular} if it vanishes on all rank 1
projections, or equivalently, if it vanishes on the compact operators. The
following theorem is our main result.

\begin{theorem}\label{maintheorem}
An uncountable cardinal $\kappa$ is
\begin{itemize}
\item[(i)] Ulam real-valued measurable if and only if there is a
singular countably additive state on $B(l^2(\kappa))$

\item[(ii)] Ulam measurable if and only if there is a
singular countably additive pure state on $B(l^2(\kappa))$

\item[(iii)] real-valued measurable if and only if there is a
singular ${<}\,\kappa$-additive state on $B(l^2(\kappa))$

\item[(iv)] measurable if and only if there is a singular
${<}\,\kappa$-additive pure state on $B(l^2(\kappa))$.
\end{itemize}
\end{theorem}

The four parts of this theorem appear in Theorems \ref{ulamrvmeas},
\ref{rvmeas}, \ref{meas}, and \ref{ulammeas}.

The real-valued variants are substantially easier, so we will present
them first.

\section{Quantum real-valued and Ulam real-valued measurability}

We can immediately prove part (i) of Theorem \ref{maintheorem}.

\begin{theorem}\label{ulamrvmeas}
Let $\kappa$ be an uncountable cardinal. Then $\kappa$ is Ulam real-valued
measurable if and only if there is a singular countably additive state on
$B(l^2(\kappa))$.
\end{theorem}

\begin{proof}
($\Leftarrow$) \ Suppose $\phi$ is a countably additive state on
$B(l^2(\kappa))$ which vanishes on the compacts. Identifying the diagonal
subalgebra with $l^\infty(\kappa)$, the restriction of $\phi$ to
this subalgebra is a countably additive state on $l^\infty(\kappa)$ which
vanishes on all finitely supported functions. Thus $\kappa$ is Ulam
real-valued measurable by Proposition \ref{measvsstate} (i).

($\Rightarrow$) \ Suppose $\mu$ is a countably additive probability
measure on $l^\infty(\kappa)$ which vanishes on singletons,
and let $\{e_\alpha: \alpha \in \kappa\}$ be the standard
basis of $l^2(\kappa)$. Then the map $E$ taking $A \in B(l^2(\kappa))$
to the sequence $(\langle Ae_\alpha,e_\alpha\rangle)$ is a normal conditional
expectation onto the diagonal subalgebra, which we can again identify with
$l^\infty(\kappa)$. According to Proposition \ref{measvsstate} (i),
integration against $\mu$ is a countably additive state on
$l^\infty(\kappa)$ that vanishes on finitely supported functions,
so $\phi: x \mapsto \int E(x)\, d\mu$ is a state on $B(l^2(\kappa))$.
Since $E$ takes compact operators into $c_0(\kappa)$, $\phi$ vanishes on
the compacts. Finally, if $\{p_n: n \in \mathbb{N}\}$ is any family of
countably many mutually orthogonal projections in $B(l^2(\kappa))$, then the
functions $f_n = E(p_n)$ are positive, and using normality of $E$ and
the monotone convergence theorem, we have
\begin{eqnarray*}
\phi\left(\sum p_n\right)
&=& \int E\left(\sum p_n\right)\, d\mu
\,=\, \int \left(\sum f_n\right)\, d\mu\cr
&=& \sum \int f_n\, d\mu
\,=\, \sum \phi(p_n).
\end{eqnarray*}
This shows that $\phi$ is countably additive.
\end{proof}

The proof of Theorem \ref{maintheorem} (iii) is almost identical to this.
There are only two differences: first, we invoke part (iii) of Proposition
\ref{measvsstate} instead of part (i), and second, in the final computation
we require the following transfinite version of the monotone convergence
theorem.

\begin{lemma}\label{summing}
Let $\kappa$ be an uncountable cardinal and suppose $\mu$ is a
${<}\,\kappa$-additive probability measure on (all subsets of) some set $X$.
Let $\lambda < \kappa$. If $\{f_\alpha: \alpha \in \lambda\}$ is any family of
$\lambda$ many positive elements of $l^\infty(X)$ then
$\int \sum f_\alpha\, d\mu = \sum \int f_\alpha\, d\mu$. 
\end{lemma}

\begin{proof}
Let $f = \sum f_\alpha$. Since $\sum_{\alpha \in S} f_\alpha \leq f$ for any
finite subset $S \subset \lambda$, we immediately have
$$\sum_{\alpha \in S} \int f_\alpha\, d\mu
= \int \sum_{\alpha \in S} f_\alpha\, d\mu \leq \int f\, d\mu,$$
and taking the supremum over $S$ yields
$\sum_{\alpha \in \lambda} \int f_\alpha\, d\mu \leq \int f\, d\mu$.
Conversely, for each $\alpha \in \lambda$ let
$g_\alpha = \sum_{\beta < \alpha} f_\beta$. Inductively assuming that the
lemma holds for families of any smaller cardinality than $\lambda$, we get
$\int g_\alpha\, d\mu = \sum_{\beta < \alpha} \int f_\beta\, d\mu$
for each $\alpha \in \lambda$. So it will suffice to show that
$\lim\int g_\alpha\, d\mu \geq \int f\, d\mu$.

Let $h$ be a simple function satisfying $0 \leq h \leq f$ and let
$\epsilon > 0$. We claim that $\lim\int g_\alpha\, d\mu \geq 
\int h\, d\mu - \epsilon(1 + \|h\|_\infty)$. Since $\epsilon$ is arbitrary
this yields that $\lim\int g_\alpha\, d\mu \geq \int h\, d\mu$, and taking
the supremum over $h$ then yields $\lim\int g_\alpha\, \mu \geq
\int f\, d\mu$.

To prove the claim, for each $\alpha \in \lambda$ define
$$A_\alpha = \{x \in X: g_\alpha(x) > h(x) - \epsilon\}.$$
Then the sets $A_{\alpha+1}\setminus A_\alpha$ partition $X$ and so
${<}\, \kappa$-additivity yields
$$\sum_{\alpha \in \lambda} \mu(A_{\alpha+1}\setminus A_\alpha) = 1.$$
So we can find $\alpha_0 \in \lambda$ such that
$$\mu(A_{\alpha_0}) = 
\sum_{\beta < \alpha_0} \mu(A_{\beta + 1}\setminus A_\beta)
\geq 1 - \epsilon.$$
Then $g_{\alpha_0} > h - \epsilon$ on $A_{\alpha_0}$, and
$\mu(X\setminus A_{\alpha_0}) \leq \epsilon$ implies
$\int_{X\setminus A_{\alpha_0}} h\, d\mu
\leq \epsilon\cdot\|h\|_\infty$, so we have
$$\int_X g_{\alpha_0}\, d\mu \geq \int_{A_{\alpha_0}} g_{\alpha_0}\, d\mu
\geq \int_{A_{\alpha_0}} h\, d\mu - \epsilon
\geq \int_X h\, d\mu - \epsilon(1 + \|h\|_\infty),$$
which proves the claim.
\end{proof}

Thus we obtain

\begin{theorem}\label{rvmeas}
Let $\kappa$ be an uncountable cardinal. Then $\kappa$ is real-valued
measurable if and only if there is a singular ${<}\,\kappa$-additive state
on $B(l^2(\kappa))$.
\end{theorem}

\section{Countably additive pure states}

The correspondence between $\{0,1\}$-valued measures on $\kappa$ and pure
states on $B(l^2(\kappa))$ described in Theorem \ref{maintheorem} (ii) and
(iv) would be easy to establish if {\em Anderson's conjecture} were true.
This conjecture asserts that any pure state on $B(H)$ arises by
composing a pure state on some atomic masa with the conditional expectation
onto that masa \cite{And}. If this were true then the technique of Theorems
\ref{ulamrvmeas} and \ref{rvmeas} could be straightforwardly used to prove
Theorem \ref{maintheorem} (ii), (iv). Unfortunately,  the conjecture is known
to be false for any infinite dimensional $H$ if the continuum hypothesis is
true \cite{AkWeav}, and also under various other weaker conditions \cite{FW}.
However, it is expected to be consistent with ZFC, and presumably also with
the existence of measurable cardinals.

In order to prove Theorem \ref{maintheorem} (ii) and (iv) without invoking
this questionable assumption, we need to develop the theory of
countably additive pure states. The first property we need is regularity.
A state $\phi$ on a von Neumann algebra $M$ is {\em regular} if
$$\phi(q_n) = 0\mbox{ for all }n\qquad\Rightarrow\qquad
\phi\left(\bigvee q_n\right) = 0$$
for any countable set of projections $\{q_n\}$ in $M$. A good source on
this topic is \cite{Ham}. The basic facts we need are these:

\begin{theorem}\label{pureregular}
\cite[Proposition 10.1.5, Theorem 10.3.7]{Ham} 
Every countably additive state on a von Neumann algebra is regular. Every
regular pure state on a von Neumann algebra is countably additive.
\end{theorem}

We will state the results in this section for an arbitrary von Neumann
algebra $M$, but all that is really needed for Theorem \ref{maintheorem} is
the case where $M = B(H)$.

According to \cite{KK}, it is consistent that there can be regular (non-pure)
states which are not countably additive. We will come back to this point in
Section 8 when we discuss regularity in greater detail.
Theorem \ref{pureregular} is all we need for now.

Next, we require Farah and Weaver's theory of quantum filters \cite{FW}
(see also \cite{Bice}). A {\em quantum filter} on a von Neumann algebra
$M$ is a family of projections $\mathcal{F}$ in $M$ with the properties
\begin{itemize}
\item[(i)] if $p \in \mathcal{F}$ and $p \leq q$ then $q \in \mathcal{F}$

\item[(ii)] if $p_1, \ldots, p_n \in \mathcal{F}$ then $\|p_1\cdots p_n\| = 1$.
\end{itemize}
If $\phi$ is a state on $M$ then
$$\mathcal{F}_\phi = \{p \in M: p\mbox{ is a projection and }\phi(p) = 1\}$$
is a quantum filter; this is easy to see by working in the GNS representation
for $\phi$ (if $\phi(p_i) = 1$ then $\pi_\phi(p_i)v = v$ where $\pi_\phi$
is the GNS representation and $v$ is the implementing unit vector). Indeed,
every maximal quantum filter arises in this way from a pure state, and this
sets up a 1-1 correspondence between pure states and maximal quantum filters
\cite{FW}. However, we will not need this full result. All we need is the fact
that $\mathcal{F}_\phi = \mathcal{F}_\psi$ implies $\phi = \psi$ when $\phi$ is
pure. This is stated in a footnote in \cite{FW} but a proof is not given, so
we include a proof here.

\begin{lemma} \label{vnpu}
Suppose that $\phi$ is a pure state on a von Neumann algebra $M$ and that
$\psi$ is a state on $M$ such that
$$\phi(p) = 1 \qquad\Rightarrow\qquad \psi(p) = 1$$
for any projection $p \in M$. Then $\phi = \psi$.
\end{lemma}

\begin{proof} The {\em multiplicative domain} of $\phi$ is the set
$D = \{x \in M: \phi(xy) = \phi(yx) = \phi(x)\phi(y)$ for all $y \in M\}$.
A projection $p$ belongs to $D$ if and only if $\phi(p) = 0$ or $1$
\cite[Section 1]{And1}. Now we are given that $\phi(p) = 1$ implies
$\psi(p) = 1$, and thus also
$$\phi(p) = 0 \quad\Rightarrow\quad \phi(1 - p) = 1
\quad\Rightarrow\quad \psi(1 - p) = 1
\quad\Rightarrow\quad \psi(p) = 0.$$
So $\phi(p) = \psi(p)$ for every projection $p \in D$. But the
span of the projections in $D$ is norm dense in $D$
\cite[Proposition on p.\ 305]{And1}, so by linearity and continuity
$\phi$ and $\psi$ agree on $D$. Finally, since $\phi$ is pure it is
the unique state extension to $M$ of its restriction to $D$
\cite[Corollary on p.\ 307]{And1}. Since $\phi$ and $\psi$ agree on $D$,
this entails that $\phi = \psi$.
\end{proof}

If $\phi$ is regular then quantum filters become especially nice. Note
that in general, quantum filters need not be actual filters in the usual
sense of being stable under finite meets.

\begin{lemma}\label{qfilter}
Let $\phi$ be a regular state on a von Neumann algebra $M$. Then
$\mathcal{F}_\phi$ is a $\sigma$-filter, i.e., it is stable under countable
meets.
\end{lemma}

\begin{proof}
Let $\{p_n\}$ be a countable family in $\mathcal{F}_\phi$ and write
$p_n^\perp = 1 - p_n$. Then
$$\phi\left(\bigwedge p_n\right) = \phi\left(1 - \bigvee p_n^\perp\right)
= 1 - \phi\left(\bigvee p_n^\perp\right) = 1$$
where $\phi(\bigvee p_n^\perp) = 0$ by regularity. Thus
$\bigwedge p_n \in \mathcal{F}_\phi$.
\end{proof}

Recall that for pure states, regularity is equivalent to countable
additivity (Theorem \ref{pureregular}). Our focus will now turn exclusively
to pure states for the remainder of this section.

Given a pure state $\phi$ on $M$, say that a projection
$p \in \mathcal{F}_\phi$ {\em isolates} a positive element $x \in M$ if
$\phi(x) = \|pxp\|$. Note that we automatically have
$\phi(x) = \phi(pxp) \leq \|pxp\|$ for any $x \geq 0$ and
$p \in \mathcal{F}_\phi$; the first equality
follows from the fact that $p$ belongs to the multiplicative domain of
$\phi$ (cf.\ the proof of Lemma \ref{vnpu}), or it can be directly checked
by passing to the GNS representation for $\phi$. It follows that if
$p \in \mathcal{F}_\phi$ isolates $x$, then so does any $p' \leq p$ in
$\mathcal{F}_\phi$, since $\|p'xp'\| \leq \|pxp\| = \phi(x)$, and as we
just saw, the reverse inequality is automatic.

\begin{lemma}\label{isolate}
Suppose that $\phi$ is a countably additive pure state on a von Neumann
algebra $M$ and let $x \in M$ be positive.  Then there is a projection
$p \in \mathcal{F}_\phi$ which isolates $x$.
\end{lemma}

\begin{proof}
We must find $p \in \mathcal{F}_\phi$ such that $\|pxp\| \leq \phi(x)$.
Suppose $||pxp|| > a$ for all $p \in \mathcal{F}_\phi$. Then for any
finite subset $F = \{p_1, \ldots, p_n\} \subset \mathcal{F}_\phi$
let $p_F = p_1 \wedge \cdots \wedge p_n$ be their meet. According to
Lemma \ref{qfilter}, $p_F \in \mathcal{F}_\phi$.

Assume $M$ is concretely represented on a Hilbert space $H$. For each $F$,
find a unit vector $v_F \in {\rm Ran}\, p_F$ such that
$\langle xv_F, v_F\rangle = \langle (p_Fxp_F)v_F, v_F \rangle
> a$, and let $\psi_F$ be the corresponding vector state
$\psi_F: y \mapsto \langle y v_F, v_F \rangle$ on $M$.  Any weak* limit of
this net of vector states (with the finite sets $F$ ordered by inclusion)
will then be a state which takes a value $\geq a$ on $x$ and takes the value
$1$ on every projection in $\mathcal{F}_\phi$. By Lemma \ref{vnpu} this
state must equal $\phi$, so we have shown that $\phi(x) \geq  a$.

We now know that for every $n$ there exists $p_n \in \mathcal{F}_\phi$
such that $\phi(x) \geq ||p_n x p_n|| - \frac{1}{n}$. Letting
$p = \bigwedge p_n$ then yields $\phi(x) \geq ||pxp||$, and $\phi(p) = 1$
by Lemma \ref{qfilter}.
\end{proof} 

The previous lemma is actually all we need for Theorem \ref{maintheorem}, but
a stronger conclusion can be drawn.

For any $x \in M$, say that $p \in \mathcal{F}_\phi$ {\em excises} $\phi$
for $x$ if $pxp = \phi(x)p$. This is different from, but related to, the
notion of excision in \cite{AAP}. Note again that if $p \in \mathcal{F}_\phi$
excises $\phi$ for $x$ then so does any $p' \leq p$ in $\mathcal{F}_\phi$,
since $$p'xp' = p'(pxp)p' = \phi(x)p'pp' = \phi(x)p'.$$

\begin{lemma}\label{excision}
Suppose that $\phi$ is a countably additive pure state on a von Neumann
algebra $M$ and let $x \in M$.  Then there is a projection
$p \in \mathcal{F}_\phi$ which excises $\phi$ for $x$.
\end{lemma}

\begin{proof}
Assume first that $x$ is positive, and by Lemma \ref{isolate} find
$q \in \mathcal{F}_\phi$ which isolates $x$. If $\phi(x) = 0$ then $q$
excises $\phi$ for $x$, so assume $\phi(x) > 0$. Then $\phi(qxq) = \phi(x) =
\|qxq\|$, so by \cite[Proposition on p.\ 305]{And1}, for any $n \in \mathbb{N}$
the spectral projection $p_n = P_{[\|qxq\| - 1/n,\|qxq\|]}(qxq)$ of $qxq$ for
the interval $[\|qxq\| - \frac{1}{n},\|qxq\|]$ belongs to $\mathcal{F}_\phi$.
Then $p = \bigwedge p_n \in \mathcal{F}_\phi$ by Lemma \ref{qfilter},
and since $p \leq q$ we have
$$pxp = p(qxq)p = \|qxq\|p = \phi(x)p,$$
where the middle equality holds since $p$ is the spectral projection of $qxq$
for the singleton set $\{\|qxq\|\}$. So $p$ excises $\phi$ for $x$.

Now every $x \in M$ is a linear combination of four positive elements $x_i$,
and taking the meet of four projections which excise $\phi$ for the $x_i$
will produce a projection that excises $\phi$ for $x$.
\end{proof}

We can infer a satisfying result about the continuity of countably additive
pure states.

\begin{theorem}\label{weakstar}
Let $\kappa$ be an uncountable cardinal, let $\phi$ be a ${<}\,\kappa$-additive
pure state on a von Neumann algebra $M$, and let $N$ be a von Neumann
subalgebra of $M$ that is generated by fewer than $\kappa$ elements. Then the
restriction of $\phi$ to $N$ is normal. In particular, any countably
additive pure state on a von Neumann algebra is sequentially weak* continuous.
\end{theorem}

\begin{proof}
If $N$ is generated by $\lambda$ many elements, where $\lambda < \kappa$
is infinite, then the set of polynomials, with complex rational coefficients,
in these generators and their adjoints still has cardinality $\lambda$ and is
weak* dense in $N$. Let $S = \{x_\alpha: \alpha \in \lambda\}$ be this set.
For each $\alpha \in \lambda$ find
$q_\alpha \in \mathcal{F}_\phi$ which excises $\phi$ for $x_\alpha$. Then let
$p_0 = 1$, $p_\alpha = \bigwedge_{\beta < \alpha} q_\beta$ for $\alpha \geq 1$,
and $p = p_\lambda = \bigwedge_{\beta \in \lambda} q_\beta$. We prove
by transfinite induction that each $p_\alpha$ belongs to $\mathcal{F}_\phi$.
At successor ordinals this follows from Lemma \ref{qfilter}. At limit
ordinals, assuming $p_\beta \in \mathcal{F}_\phi$ for all $\beta < \alpha$,
we have $\phi(p_\beta - p_{\beta + 1}) = 0$ for all such $\beta$, and
$\sum_{\beta < \alpha} (p_\beta - p_{\beta+1}) = 1 - p_\alpha$. So by
${<}\,\kappa$-additivity $\phi(1 - p_\alpha) = 0$, i.e., $\phi(p_\alpha) = 1$.
This completes the induction and we conclude that $\phi(p) = 1$, i.e.,
$p \in \mathcal{F}_\phi$. We have shown that there exists a projection
which simultaneously excises $\phi$ for every element of $S$.

Now any $x \in N$ is the weak* limit of a net $(x_i)$ in $S$,
so that $\phi(x_i)p = px_ip \to pxp$ weak*. This shows that $pxp = ap$
for some $a \in \mathbb{C}$, and in fact $\phi(x) = \phi(pxp) = \phi(ap) = a$.
So $p$ excises $\phi$ for $x$ as well. We conclude that $\phi$ excises $\phi$
for every element of $N$. Then if $(y_j)$ in any net in $N$ which converges
weak* to some $y \in N$, we have
$$\phi(y_i)p = py_ip \to pyp = \phi(y)p$$
where the middle convergence is weak*. Thus $\phi(y_i) \to \phi(y)$. This
shows that the restriction of $\phi$ to $N$ is normal.
\end{proof}

We thank Ilijas Farah for pointing out that this version of Theorem
\ref{weakstar} follows from an earlier, weaker version.

If it is consistent with ZFC that measurable cardinals exist, then it is
consistent with ZFC that $2^{\aleph_0}$ is Ulam real-valued measurable.
In this case, according to Proposition \ref{measvsstate} (i) there would
be a singular countably additive state on $l^\infty(2^{\aleph_0})$. But
$l^\infty(2^{\aleph_0})$ is countably generated as a von Neumann algebra
(for details see the proof of Lemma \ref{bigger} below), so this would be a
countably additive singular state on a countably generated von Neumann
algebra. Thus, it is consistent that Theorem
\ref{weakstar} can fail for states which are not pure.

On the other hand, if ZFC is consistent then so is ZFC + ``real-valued
measurable cardinals do not exist''. Under this hypothesis, the restriction
of any countably additive state on any von Neumann algebra to any atomic
abelian von Neumann subalgebra must be normal. (After removing its normal
part what is left would be singular and countably additive, and hence zero.)
But by Proposition \ref{addvsnormal}
this simply means that any countably additive state must be normal. So it
is also consistent with ZFC that every countably additive state is normal,
making Theorem \ref{weakstar} true for general states in a strong form.

\section{Quantum measurability}

We prove Theorem \ref{maintheorem} (iv). The forward direction uses the
following generalization of the Kadison-Singer problem to Hilbert spaces
of arbitrary dimension.

\begin{theorem}\label{kadsing}
Let $\kappa$ be an infinite cardinal. Then any pure state on the diagonal
subalgebra of $B(l^2(\kappa))$ has a unique (necessarily pure) state
extension to $B(l^2(\kappa))$.
\end{theorem}

\begin{proof}
The problem was recently solved for $\kappa = \omega$ by Marcus, Spielman,
and Srivastava \cite{MSS}. A well known equivalent version \cite[Lemma
5 on p. 395]{KS} states that for any $\epsilon > 0$, there exists
$k \in \mathbb{N}$ such that any $A \in B(l^2(\omega))$, $\|A\| \leq 1$,
with null diagonal can be $k$-paved, meaning that there exist diagonal
projections $p_1, \ldots, p_k$ such that $\sum p_i = 1$ and
$\|p_iAp_i\| \leq \epsilon$ for all $i$. Now if $\kappa > \omega$ and
$A \in B(l^2(\kappa))$, $\|A\| \leq 1$, has null diagonal, then given
$\epsilon > 0$, for any countable subset $S$ of $\kappa$ the
compression of $A$ to $l^2(S)$ (i.e., the operator $p_SAp_S$ where $p_S$
is the orthogonal projection onto $l^2(S)$) can be $k$-paved. This
$k$-paving corresponds to a $k$-coloring of the set $S$. Ordering the
countable subsets of $\kappa$ by inclusion and letting $C_S$ be the set
of $k$-colorings of $S$ which correspond to $k$-pavings of $p_SAp_S$, we
find that the sets $C_S$ are nonempty and closed in $\{1, \ldots, k\}^S$
(giving $\{1, \ldots, k\}$ the discrete topology).
Therefore, since $\{1, \ldots, k\}^\kappa$ is compact,
they have a nonempty intersection. This yields a $k$-paving
of $A$. So every operator in the unit ball of $B(l^2(\kappa))$ with null
diagonal can be paved, and this implies that pure states have unique
extensions.
\end{proof}

For the reverse direction of the upcoming theorem, we have to prove that if
there is a singular ${<}\,\kappa$-additive pure state on $B(l^2(\kappa))$
then $\kappa$ is measurable. We know from Theorem \ref{rvmeas} that $\kappa$
is real-valued measurable; to complete the proof that it is measurable,
we just have to show $\kappa > 2^{\aleph_0}$. That is the content of
the next lemma.

\begin{lemma}\label{bigger}
Let $\kappa$ be an uncountable cardinal and suppose there is a singular
countably additive pure state on $B(l^2(\kappa))$. Then
$\kappa > 2^{\aleph_0}$.
\end{lemma}

\begin{proof}


Assume $\kappa \leq 2^{\aleph_0}$. Fix an orthonormal basis
$$\{e_{b_1 b_2 \ldots}: (b_1, b_2, \ldots ) \in S\}$$
of $l^2(\kappa)$ where $S$ is some subset of $\{0,1\}^\omega$.
For each finite sequence $(b_1, \ldots, b_n) \in \{0,1\}^n$
let $p_{b_1 \ldots b_n}$ be the orthogonal projection onto
the span of all the basis vectors that share this initial segment.
This projection could be zero.

For any sequence $(b_1, b_2, \ldots)$ in $S$ the meet of the projections
$p_{b_1}$, $p_{b_1b_2}$, $\ldots$ is the rank one projection onto the span
of $e_{b_1 b_2 \ldots}$. This shows that the diagonal subalgebra of
$B(l^2(\kappa))$ is generated by the projections $p_{b_1, \ldots, b_n}$.
This is a countable family, so we conclude that the diagonal subalgebra
is countably generated as a von Neumann algebra, and it then follows
from Theorem \ref{weakstar} that any countably additive pure state on
$B(l^2(\kappa))$ must be normal on the diagonal. But any singular state
on $B(l^2(\kappa))$ restricts to a singular state on the diagonal, so
we conclude that no countably additive pure state on $B(l^2(\kappa))$ is
singular.
\end{proof}

In connection with the last proof, note that $B(l^2(\kappa))$ is not
countably generated as a von Neumann algebra for any $\kappa > \aleph_0$.
Any countably generated von Neumann algebra acting on $l^2(\kappa)$
contains a weak* dense separable C*-subalgebra. But if $A$ is any separable
C*-algebra acting on $l^2(\kappa)$, then for any nonzero $v \in l^2(\kappa)$
the set $\overline{Av}$ is a separable invariant subspace for $A$, showing
that $A'$ is nontrivial and therefore $A$ is not weak* dense in
$B(l^2(\kappa))$.

We are ready to prove Theorem \ref{maintheorem} (iv).

\begin{theorem}\label{meas}
Let $\kappa$ be an uncountable cardinal. Then $\kappa$ is measurable if and
only if there is a singular ${<}\,\kappa$-additive pure state on
$B(l^2(\kappa))$.
\end{theorem}

\begin{proof}
Suppose $\kappa$ is measurable and per Proposition \ref{measvsstate} (iv)
let $\phi$ be a singular ${<}\,\kappa$-additive pure state on
$l^\infty(\kappa)$. As in the proof of Theorem \ref{ulamrvmeas}, let
$E$ be the conditional expectation from $B(l^2(\kappa))$ onto the
diagonal subalgebra which is identified with $l^\infty(\kappa)$. Then
$\phi\circ E$ is a state on $B(l^2(\kappa))$, and as its restriction
to the diagonal masa is pure, Theorem \ref{kadsing} implies that it
is pure. It is singular and ${<}\,\kappa$-additive just as in the proofs
of Theorems \ref{ulamrvmeas} and \ref{rvmeas}.

Conversely, suppose there is a singular ${<}\,\kappa$-additive pure state
on $B(l^2(\kappa))$. According to Theorem \ref{rvmeas} this means that
$\kappa$ is real-valued measurable, and it is $> 2^{\aleph_0}$ by Lemma
\ref{bigger}. So $\kappa$ must be measurable.
\end{proof}

\section{Quantum Ulam measurability}

Theorem \ref{maintheorem} (ii) still needs proof. We already know that
if $B(l^2(\kappa))$ has a singular countably additive pure state then
$\kappa$ is Ulam real-valued measurable (Theorem \ref{ulamrvmeas}),
and hence $\geq$ the smallest real-valued measurable cardinal. We
also know from Lemma \ref{bigger} that $\kappa > 2^{\aleph_0}$. But
we cannot conclude from these two facts that $\kappa$ is $\geq$ the
smallest measurable cardinal, i.e., is Ulam measurable. That requires
an additional argument.

We need the following variation on Lemma \ref{bigger}.

\begin{lemma}\label{bigger2}
Let $\kappa$ be an uncountable cardinal and let $M$ be a von Neumann
algebra which contains $l^\infty(\kappa)$ as a (not necessarily unital)
von Neumann subalgebra. Suppose there is a countably additive pure state
on $M$ whose restriction to $l^\infty(\kappa)$ is nonzero and singular. Then
$\kappa > 2^{\aleph_0}$.
\end{lemma}

\begin{proof}
As in the proof of Lemma \ref{bigger}, this follows from the fact that
if $\kappa \leq 2^{\aleph_0}$ then $l^\infty(\kappa)$ is countably
generated, and hence the restriction to it of any countably additive pure
state on $M$ must be normal by Theorem \ref{weakstar}.
\end{proof}

%
%

The point is that we do not assume the restriction of $\phi$ to
$l^\infty(\kappa)$ is pure. All we know about $\phi|_{l^\infty(\kappa)}$
is that it is a countably additive nonzero singular state, and this does
not in itself contradict $\kappa \leq 2^{\aleph_0}$.

Again, the proof of Theorem \ref{ulammeas} below could be modified in an
obvious way so as to only require Lemma \ref{bigger2} for $M = B(H)$.

\begin{theorem}\label{ulammeas}
Let $\kappa$ be an uncountable cardinal. Then $\kappa$ is Ulam measurable
if and only if there is a singular countably additive pure state on
$B(l^2(\kappa))$.
\end{theorem}

\begin{proof}
The proof of the forward direction is virtually identical to the proof
of the forward direction of Theorem \ref{meas}, just with countable
additivity in place of ${<}\,\kappa$-additivity and using Proposition
\ref{measvsstate} (ii) instead of Proposition \ref{measvsstate} (iv).

For the reverse direction, let $\kappa$ be the smallest cardinal with
the property that there is a von Neumann algebra $M$ containing (possibly
nonunitally)
$l^\infty(\kappa)$ as a von Neumann subalgebra and a countably additive
pure state $\phi$ on $M$ whose restriction to $l^\infty(\kappa)$ is nonzero
and singular. We claim that $\phi$ must be ${<}\,\kappa$-additive on $M$. If
so, then it is ${<}\,\kappa$-additive on the embedded $l^\infty(\kappa)$
and this implies via Proposition \ref{measvsstate} (iii) that $\kappa$ is
real-valued measurable. But $\kappa > 2^{\aleph_0}$ by Lemma \ref{bigger2},
so it must in fact be measurable. Finally, any singular countably additive
pure state on $B(l^2(\kappa'))$, for any cardinal $\kappa'$, restricts to
a nonzero, singular state on the diagonal $l^\infty(\kappa')$ and thus
we must have $\kappa' \geq \kappa$, i.e., $\kappa'$ must be Ulam measurable.
This completes the proof modulo the claim.

To prove the claim, fix $M$ and $\phi$ as above and let
$\{p_\alpha: \alpha \in \lambda\}$ be a family of mutually orthogonal
nonzero projections in $M$ where $\lambda < \kappa$. We want to
show that $\phi(\sum p_\alpha) = \sum \phi(p_\alpha)$. The set of $p_\alpha$
for which $\phi(p_\alpha) \neq 0$ is countable, and by countable additivity
$\phi$ is additive on that set. Thus, by removing the indices where
$\phi(p_\alpha) \neq 0$, we can assume $\phi(p_\alpha) = 0$ for all $\alpha$.
Then the von Neumann subalgebra generated by the $p_\alpha$ is
isomorphic to $l^\infty(\lambda)$ and the restriction of $\phi$ to this
subalgebra is singular. By minimality of $\kappa$, this implies that
$\phi$ must be zero on the embedded $l^\infty(\lambda)$, i.e.,
since $\sum p_\alpha$ is the unit of this algebra, $\phi(\sum p_\alpha) = 0$.
This is what we needed to show.
\end{proof}

\section{General von Neumann algebras}

With a little more work, parts (i) and (iii) of
Theorem \ref{maintheorem} can be generalized to
arbitrary von Neumann algebras. We need the following lemma, which is
probably well known.

\begin{lemma}\label{singular}
Let $M$ be a von Neumann algebra with a von Neumann subalgebra $N$ and
a normal conditional expectation $E: M \to N$. Then the image under
$E^{**}$ of the singular projection for $M$ is the singular projection
for $N$.
\end{lemma}

\begin{proof}
Let $z_M$ and $z_N$ be the two singular projections. If
$E_*: N_* \to M_*$ is the predual map  and $i_N: N_* \to N^*$ is the
canonical embedding, then $E^*i_N = i_ME_*$. Dualizing, we have
$Ei_M^* = i_N^*E^{**}$. It follows that $E^{**}({\rm ker}\, i_M^*)
\subseteq {\rm ker}\, i_N^*$. So $E^{**}(z_M) \leq z_N$. Letting
$i$ be the inclusion of $N$ into $M$, we also have $i_M^*i^{**} = ii_N^*$,
so that $i^{**}({\rm ker}\, i_N^*) \subseteq {\rm ker}\, i_M^*$. Applying
$E^{**}$ yields ${\rm ker}\, i_N^* \subseteq E^{**}({\rm ker}\, i_M^*)$, and
hence $z_N \leq E^{**}(z_M)$. So $E^{**}(z_M) = z_N$.
\end{proof}

We also need the fact that for any uncountable $\kappa$, the smallest cardinal
that supports a ${<}\,\kappa$-additive probability measure which vanishes on
singletons is real-valued measurable; this generalizes the standard fact that
the smallest Ulam real-valued measurable cardinal is real-valued measurable,
and is proven in the same way. Similarly, the smallest cardinal that supports
a ${<}\,\kappa$-additive $\{0,1\}$-valued measure which vanishes on
singletons is measurable.

\begin{theorem}\label{vnalg1}
Let $\kappa$ be an uncountable cardinal. Then a von Neumann algebra $M$
possesses a ${<}\,\kappa$-additive singular state if and only if $M$
is not ${<}\,\kappa'$-decomposable, where $\kappa'$ is the smallest
real-valued measurable cardinal $\geq \kappa$.
\end{theorem}

\begin{proof}
Suppose there is a ${<}\,\kappa$-additive singular state $\phi$ on $M$.
By Zorn's lemma, we can find a family $\{p_\alpha: \alpha \in \lambda\}$
of mutually orthogonal projections in ${\rm ker}\,\phi$ whose sum is 1
\cite[p.\ 53]{AA}. Then
$$\mu(S) = \phi\left(\sum_{\alpha \in S} p_\alpha\right)$$
defines a ${<}\,\kappa$-additive measure on $\lambda$ which vanishes on
singletons. By the comment made before the theorem, the smallest cardinal
that supports a ${<}\,\kappa$-additive measure which vanishes on
singletons must be real-valued measurable, and is trivially $\geq \kappa$;
this shows that $\lambda \geq \kappa'$. So $M$ is not
${<}\,\kappa'$-decomposable.

For the reverse direction, suppose $M$ is not ${<}\,\kappa'$-decomposable.
Then there is a family
$\{p_\alpha: \alpha \in \kappa'\}$ of mutually orthogonal nonzero projections
in $M$ with sum $1$. Set $N = \bigoplus p_\alpha M p_\alpha$. For each
$\alpha$ let $\psi_\alpha$ be a normal state on $p_\alpha M p_\alpha$;
then applying these states to the summands of $N$ yields a normal map
$\omega: N \to l^\infty(\kappa')$. Let $E: M \to N$ be the canonical normal
conditional expectation, and by Proposition \ref{measvsstate} (iii) let
$\phi$ be a ${<}\,\kappa'$-additive (and hence ${<}\,\kappa$-additive)
singular state on $l^\infty(\kappa')$ given by integration against
a ${<}\,\kappa$-additive probability measure $\mu$ which vanishes on
singletons. We claim that $\rho = \phi\circ \omega\circ E$ is a
${<}\,\kappa$-additive singular state on $M$.

It is clear that $\rho$ is a state, and it is ${<}\,\kappa$-additive
because $\phi$ is ${<}\,\kappa$-additive and $\omega$ and $E$ are normal.
To see that $\rho$ is singular, identify $l^\infty(\kappa')$ with the von
Neumann algebra generated by the $p_\alpha$. Then $\Psi = \omega\circ E$ is
a normal conditional expectation from $M$ onto $l^\infty(\kappa')$. So if
$z$ is the singular projection for $M$ then
$$\rho(z) = \phi(\Psi^{**}(z)) = 1$$
since by Lemma \ref{singular} $\Psi^{**}(z)$ is the singular projection
for $l^\infty(\kappa')$. It follows that $\rho$ is singular.
\end{proof}

Using the fact that $B(l^2(\kappa))$ is not ${<}\,\kappa'$-decomposable
if and only if $\kappa' \leq \kappa$, it is not hard to deduce Theorem
\ref{maintheorem} (i) and (iii) from Theorem \ref{vnalg1}.

If there is no real-valued measurable cardinal $\geq \kappa$, then Theorem
\ref{vnalg1} becomes the assertion that no von Neumann algebra possesses a
${<}\,\kappa$-additive singular state. Similarly for Theorem \ref{vnalg2}
below.

By analogy with Theorem \ref{vnalg1}, one might expect that a von
Neumann algebra possesses a ${<}\,\kappa$-additive singular pure state
if and only if it is not ${<}\,\kappa'$-decomposable, where $\kappa'$
is the smallest measurable cardinal $\geq \kappa$. However, this is not
correct (unless there are no measurable cardinals). Indeed, for any
uncountable $\kappa$ the von Neumann algebra
$M = L^\infty[0,1]\otimes l^\infty(\kappa)$ is not ${<}\,\kappa$-decomposable,
yet it does not even possess a countably additive singular pure state.
That is because it contains $L^\infty[0,1]$ as a unital von Neumann
subalgebra, so any such state on $M$ would restrict to a state on
$L^\infty[0,1]$ with the same properties, which is impossible. (In the
abelian setting, the restriction of any pure state to a unital von Neumann
subalgebra is again pure.) We do not know whether there can be a countably
additive pure state on any nonatomic von Neumann algebra, so the next result
only goes in one direction.

\begin{theorem}\label{vnalg2}
Let $\kappa$ be an uncountable cardinal. If a von Neumann algebra $M$
possesses a ${<}\,\kappa$-additive singular pure state then it
contains a family of $\kappa'$ many mutually orthogonal minimal projections,
where $\kappa'$ is the smallest measurable cardinal $\geq \kappa$.
\end{theorem}

\begin{proof}
Let $\phi$ be a ${<}\,\kappa$-additive singular pure state on $M$.
As in the proof of Theorem \ref{vnalg1}, we can find a family
$\{p_\alpha: \alpha \in \lambda\}$ of mutually orthogonal projections
such that the restriction of $\phi$ to the von Neumann algebra they generate,
which is isomorphic to $l^\infty(\lambda)$, is nonzero, ${<}\,\kappa$-additive,
and singular. Suppose $\lambda$ is the smallest cardinal with this property,
i.e., the restriction of $\phi$ to the von Neumann algebra generated by any
family of fewer than $\lambda$ mutually orthogonal projections is zero or
nonsingular. By Lemma \ref{bigger2}, we have $\lambda > 2^{\aleph_0}$,
and as in the proof of Theorem \ref{ulammeas}, $\phi$ must be
${<}\,\lambda$-additive on $M$. Thus the restriction of $\phi$ to the
embedded $l^\infty(\lambda)$ is ${<}\,\lambda$-additive, so $\lambda$ must
be real-valued measurable, and since it exceeds $2^{\aleph_0}$, measurable.
Also, it is clear that $\lambda \geq \kappa$. So $\lambda \geq \kappa'$ and
this shows that $M$ is not ${<}\,\kappa'$-decomposable.
\end{proof}

Although this result only gives one implication, it is easily seen to
imply the hard directions of Theorem \ref{maintheorem} (ii) and (iv),
in the same way that Theorem \ref{vnalg2} implied Theorem \ref{maintheorem}
(i) and (iii).

Theorems \ref{vnalg1} and \ref{vnalg2} also have the following consequence.

\begin{corollary}
Let $M$ be a von Neumann algebra.
\begin{itemize}
\item[(i)] If $M$ is ${<}\,\kappa$-decomposable, where $\kappa$ is the
first real-valued measurable cardinal (or if there are no real-valued
measurable cardinals), then every countably additive state on $M$ is normal.
\item[(ii)] If $M$ is ${<}\,\kappa'$-decomposable, where $\kappa'$ is the
first measurable cardinal (or if there are no measurable cardinals), then
every countably additive pure state on $M$ is normal.
\end{itemize}
In particular, this holds for sequentially weak* continuous states.
\end{corollary}

\begin{proof}
Assume $M$ is ${<}\,\kappa$-decomposable and let $\phi$ be a countably additive
state on $M$. If $\phi$ is not normal then we can write
$\phi = a\phi_n + b\phi_s$ where $a,b \geq 0$,
$a + b = 1$, $\phi_n$ is a normal state and $\phi_s$ is a singular state.
Since any normal state is countably additive, it follows that $\phi_s$ is
also countably additive. Thus $M$ has a countably additive singular state,
contradicting Theorem \ref{vnalg1}. We conclude that $\phi$ must have
been normal. This proves part (i); part (ii) similarly follows from
Theorem \ref{vnalg2}, using the fact that any pure state is either
normal or singular.
\end{proof}

The implication ``sequentially weak* continuous $\Rightarrow$ normal''
contained in part (i) was shown in \cite[Theorem 2.12]{Neufang}. (Note
that in that paper, ``measurable'' means ``real-valued measurable''.)

\section{Regularity}

We return to the subject of regularity introduced in Section 4. Let us
start with an alternative characterization which will be useful later.
It is an easy consequence of similar results in \cite{Ham}.

\begin{lemma}\label{altregular}
Let $\phi$ be a state on a von Neumann algebra $M$. Then $\phi$ is
regular if and only if the positive part of its kernel is weak* sequentially
closed.
\end{lemma}

\begin{proof}
Suppose $\phi$ is regular and let $(x_n)$ be a sequence of positive elements
in ${\rm ker}\, \phi$ which converges weak* to $x \in M$. It follows from
\cite[Proposition 10.1.5 (iv)]{Ham} that the support projection of each $x_n$
lies in ${\rm ker}\, \phi$, so by regularity the join of these support
projections also belongs to ${\rm ker}\,\phi$. But this join contains the
support projection of $x$, so that $\phi(x)$ must be zero. This proves the
forward direction.

Conversely, suppose the positive part of ${\rm ker}\, \phi$ is weak*
sequentially closed. If $(p_n)$ is any sequence of projections in
${\rm ker}\, \phi$ which increases to a projection $p \in M$, then it
immediately follows that $p \in {\rm ker}\, \phi$. According to
\cite[Proposition 10.1.5 (iii)]{Ham}, this implies that $\phi$ is regular.
\end{proof}

A regular state may be seen as the quantum analog of a finitely additive
measure with the property that the union of any countable family of null
sets is null. Accordingly, let us say that a finitely additive measure is
{\em regular} if the union of any countable family of null sets is null.
The methods of Sections 2 and 3 easily yield the following equivalence.

\begin{theorem}\label{regular}
Let $\kappa$ be an uncountable cardinal. Then the following are equivalent:
\begin{itemize}
\item[(i)] there is a finitely additive regular probability measure on
$\kappa$ which vanishes on singletons
\item[(ii)] there is a regular singular state on $l^\infty(\kappa)$
\item[(iii)] there is a regular singular state on $B(l^2(\kappa))$.
\end{itemize}
\end{theorem}

\begin{proof}
The proof of the equivalence (i) $\Leftrightarrow$ (ii) is just like the
proof of Proposition \ref{measvsstate} (i). The implication (iii) $\Rightarrow$
(ii) follows by restricting to the diagonal subalgebra, just as in the
reverse directions of Theorems \ref{ulamrvmeas} and \ref{rvmeas}. For the
implication (i) $\Rightarrow$ (iii), suppose $\mu$ is a finitely additive
regular probability measure on $\kappa$ which vanishes on singletons.
As in the proof of the forward
direction of Theorem \ref{ulamrvmeas}, let $E$ be the conditional
expectation from $B(l^2(\kappa))$ onto $l^\infty(\kappa)$ and define
$\phi: x \mapsto \int E(x)\, d\mu$. As in Theorem \ref{ulamrvmeas},
$\phi$ is a singular state on $B(l^2(\kappa))$. To see that it is regular,
let $\{q_n\}$ be a countable family of projections in its kernel; then for
each $n$ and $k$ the set where $E(q_n) \geq \frac{1}{k}$ is $\mu$-null,
and taking the union over $k$ yields that the support of each $E(q_n)$
in $\kappa$ is $\mu$-null. The support of $E(\bigvee q_n)$ is then a
countable union of null sets and therefore also null. It can therefore be
uniformly approximated by simple function which are supported on the same
null set, and whose integral is therefore zero. This shows that
$\phi(\bigvee q_n) = 0$.
\end{proof}

According to Theorem \ref{pureregular}, every countably additive state
is regular, and the converse is valid for pure states. It was shown in
\cite{KK} that it is consistent that the converse fails for general
states, even in the atomic abelian case. Indeed, assuming ZFC +
``a measurable cardinal exists'' is consistent, \cite{KK} shows that
there is a model in which there is no real-valued measurable cardinal
$\leq 2^{\aleph_0}$, and hence no Ulam real-valued measurable cardinal
in that range, but there is a regular measure on $\kappa = 2^{\aleph_0}$.
Together with Theorems \ref{regular} and \ref{ulamrvmeas} and
Proposition \ref{measvsstate} (i), this yields the following result.

\begin{theorem}\label{kkthm}
If it is consistent that a measurable cardinal exists, then it is consistent
that both $l^\infty(\kappa)$ and $B(l^2(\kappa))$ have regular singular
states, but neither algebra admits a countably additive singular state,
for $\kappa = 2^{\aleph_0}$.
\end{theorem}

On the other hand, it is worth noting that this conclusion is not consistent
with the continuum hypothesis. It follows from \cite[Lemma 6.2]{AA} that if the
continuum hypothesis holds and $\kappa$ is smaller than the first measurable
cardinal then there is no finitely additive regular probability measure on
$\kappa$ which vanishes on singletons. Thus, using Theorem \ref{regular} we
can draw the following conclusion.

\begin{theorem}\label{chthm}
Assuming the continuum hypothesis, $l^\infty(\kappa)$ and $B(l^2(\kappa))$
have regular singular states if and only if $\kappa$ is Ulam measurable.
\end{theorem}

Finally, there are two more things we can say without any set-theoretic
assumptions.

\begin{theorem}\label{regularsingular}
Let $\kappa$ be an infinite cardinal.
\begin{itemize}
\item[(i)] If $\kappa = \aleph_1$ then neither $l^\infty(\kappa)$ nor
$B(l^2(\kappa))$ has a regular singular state.

\item[(ii)] If $\kappa$ is Ulam real-valued measurable then both
$l^\infty(\kappa)$ and $B(l^2(\kappa))$ have regular singular states.
\end{itemize}
\end{theorem}

Part (ii) is Proposition \ref{measvsstate} (i) and Theorem \ref{ulamrvmeas},
plus the fact that countably additive implies regular. For part (i) combine
Theorem \ref{regular} with the proof of \cite[Lemma 10.13]{Jech}; that result
is stated for countably additive measures, but regularity is all that is used.

\section{Ueda's peak set theorem}

A {\em peak projection} for a von Neumann algebra $M$ is a projection
$q \in M^{**}$ which is the weak* limit in $M^{**}$ of the sequence
$(a^n)$ for some positive contraction $a$ in $M$; equivalently, it is the
meet in $M^{**}$ of a countable set of projections in $M \subseteq M^{**}$
\cite{blueda}. (In \cite{blueda} this is stated for a countable decreasing
sequence of projections, but the proof given there does not need this
restriction.) It is {\em singular} if $\psi(q) = 0$
for every normal state $\psi$ on $M$. The von Neumann algebra version of
Ueda's peak set theorem states that for every singular state
$\phi$ on $M$ there is a singular peak projection $q$ with $\phi(q) = 1$.
It was shown in \cite{blueda} that there are counterexamples to this
statement if measurable cardinals exist. We will refine this conclusion.
We thank Louis Labuschagne for conversations related to this section.

The following alternative characterization will be helpful. We retain the
notation $\mathcal{F}_\phi$ from Section 4.

\begin{lemma}\label{altueda}
Ueda's theorem fails for a von Neumann algebra $M$ if and only if there is
a singular state $\phi$ on $M$ such that the meet (in $M$) of any
countable set of projections in $\mathcal{F}_\phi$ is nonzero.
\end{lemma}

\begin{proof}
If $q \in M^{**}$ is a peak projection then it is the meet in $M^{**}$ of
a countable set of projections $\{p_n\}$ in $M$, and if $\phi(q) = 1$ then
we must have $\phi(p_n) = 1$ for all $n$. Thus $q$ is the meet in $M^{**}$
of a countable set of projections in $\mathcal{F}_\phi$. Conversely, if $q$
is the meet in $M^{**}$ of a countable set of projections $\{p_n\}$ in
$\mathcal{F}_\phi$ then it is a peak projection, and $\phi(p_n) = 1$ for all
$n$ implies that each $p_n$ dominates the support projection for $\phi$ in
$M^{**}$, so that $q$ must as well. Thus $\phi(q) = 1$. We have shown that
a projection $q \in M^{**}$ is a peak projection with $\phi(q) = 1$ if and
only if there is a countable set $\{p_n\}$ in $\mathcal{F}_\phi$ whose meet
in $M^{**}$ is $q$. In this case $q$ will be singular if and only if the
meet of the $p_n$ in $M$ is zero. (If the meet in $M$ is zero then any normal
state $\psi$ will satisfy $\psi(p_1 \wedge \cdots \wedge p_n) \to 0$
as $n \to \infty$, and hence $\psi(q) = 0$; conversely, if the meet in $M$
is not zero then there is a normal state which does not vanish on this meet
and hence does not vanish on $q$. Note that
$\psi(p_1 \wedge \cdots \wedge p_n) = \lim_{k \to \infty}
\psi( (p_1\cdots p_n\cdots p_1)^k)$ \cite{Prager}, so that if $\psi$ is
normal then $\psi(p_1\wedge \cdots \wedge p_n)$ has the same
value regardless of whether the meet is taken in $M$ or $M^{**}$.) Thus,
to say that there is a singular peak projection $q$ with $\phi(q) = 1$ is
to say that there is a countable set in $\mathcal{F}_\phi$ whose meet in
$M$ is zero.
\end{proof}

The condition given in Lemma \ref{altueda} is closely related to regularity.

\begin{proposition}\label{mcip}
Let $\phi$ be a state on a von Neumann algebra $M$. The following are
equivalent:
\begin{itemize}
\item[(i)] the meet of any countable set of projections in
$\mathcal{F}_\phi$ is nonzero

\item[(ii)] there is a projection $p \in M$ with $\phi(p) > 0$ such that
the state $\psi: x \mapsto \frac{1}{\phi(p)}\phi(pxp)$ is regular.
\end{itemize}
Moreover, if $\phi$ is singular then so is $\psi$.
\end{proposition}

\begin{proof}
(i) $\Rightarrow$ (ii) \ Suppose the meet of any countable set of projections
in $\mathcal{F}_\phi$ is nonzero. Let $a = \inf\{\phi(q): q$ is the meet
of a countable set of projections in $\mathcal{F}_\phi\}$. Then for each
$m \in \mathbb{N}$ we can find a countable set of projections $\{q^m_n\}$ in
$\mathcal{F}_\phi$ with satisfies $\phi(\bigwedge_n q^m_n) \leq
a + \frac{1}{m}$. The union of these families will be a countable set
of projections $\{p_n\}$ whose meet $p$ satisfies $\phi(p) = a$. Thus, the
infimum is achieved.

If $a = 0$ then the projections $p_n - p$ will belong to $\mathcal{F}_\phi$
and have meet zero, contradicting the hypothesis. So $a > 0$, and we need
only show that $\psi: x \mapsto \frac{1}{\phi(p)}\phi(pxp)$ is regular. We
first claim that the restriction of $\phi$ to $pMp$ is regular.
Suppose not; then there is a countable set of projections $\{r_n\}$ in
$pMp \cap {\rm ker}\, \phi$ whose join does not belong to ${\rm ker}\, \phi$.
But then the projections $p_n - r_n$ belong to $\mathcal{F}_\phi$ while
$$\phi\left(\bigwedge(p_n - r_n)\right)
= \phi\left(p - \bigvee r_n\right) < a,$$
a contradiction. This proves the claim.

We invoke Lemma \ref{altregular} to show that $\psi$ is regular. Let
$(x_n)$ be a sequence of positive elements in ${\rm ker}\, \psi$ which
converges weak* to $x \in M$. Then $\phi(px_np) = 0$ for all $n$ and
$px_np \to pxp$ weak*, so one direction of Lemma \ref{altregular} plus
the fact that $\phi$ is regular on $pMp$ shows
that $\phi(pxp) = 0$, i.e., $\psi(x) = 0$. Thus the positive part of the
kernel of $\psi$ is weak* sequentially closed, and the other direction of
Lemma \ref{altregular} then shows that $\psi$ is regular.

(ii) $\Rightarrow$ (i) \ Suppose there is a projection $p \in M$ with
$\phi(p) > 0$ such that $\psi: x \mapsto \frac{1}{\phi(p)}\phi(pxp)$ is
regular and let $\{q_n\}$ be any countable set of projections in
$\mathcal{F}_\phi$. Then for each $n$
$$\phi(p \wedge q_n) = 1 - \phi(p^\perp \vee q_n^\perp) \geq
1 - \phi(p^\perp + q_n^\perp) = 1 - \phi(p^\perp) = \phi(p)$$
where, as before, $p^\perp = 1 - p$. Thus $\psi(q_n) = 1$ for all $n$,
and by regularity of $\psi$ we therefore have $\psi(\bigwedge q_n) = 1$.
So certainly $\bigwedge q_n \neq 0$.

Finally, if $z$ is the singular projection for $M$ then applying Lemma
\ref{singular} to the conditional expectation $x \mapsto pxp$ from $M$ onto
$pMp$ yields that $pz$ is the singular projection for $pMp$.  Thus if $\phi$
is singular then $\psi(pz) = \frac{1}{\phi(p)}\phi(pz) = 1$, so $\psi$ is
also singular.
\end{proof}

We can now sharpen the result from \cite{blueda} mentioned earlier.

\begin{theorem}\label{ueda}
Let $M$ be a von Neumann algebra. The following are equivalent:
\begin{itemize}
\item[(i)] Ueda's theorem fails for $M$

\item[(ii)] $M$ possesses a regular singular state

\item[(iii)] $M$ is not ${<}\,\kappa$-decomposable where $\kappa$ is the
first cardinal that admits a finitely additive regular probability
measure which vanishes on singletons.
\end{itemize}
\end{theorem}

\begin{proof}
(i) $\Leftrightarrow$ (ii) \ The forward direction follows from Lemma
\ref{altueda} and Proposition \ref{mcip}, and the reverse direction follows
immediately from Lemma \ref{altueda}, since if $\phi$ is regular the meet
of any countable set of projections in $\mathcal{F}_\phi$ belongs to
$\mathcal{F}_\phi$ (Lemma \ref{qfilter}).

(ii) $\Rightarrow$ (iii) \ We prove the contrapositive. Let $\kappa$ be as
in part (iii) and suppose $M$ is ${<}\,\kappa$-decomposable. Let $\phi$ be
a regular state on $M$; we will show it must be completely additive and
therefore normal. To see this, let $\{p_\alpha: \alpha \in \lambda\}$ be
any family of mutually orthogonal nonzero projections in $M$. The von Neumann
subalgebra they generate is isomorphic to $l^\infty(\lambda)$, and the
restriction of $\phi$ to this embedded $l^\infty(\lambda)$ is given by
integration against a regular finitely additive measure $\mu$ on $\lambda$.
Letting $\nu$ be the unique completely additive measure which agrees with
$\mu$ on singletons, we infer that $\mu - \nu$ is a regular finitely
additive measure which vanishes on singletons, and since $\lambda < \kappa$
this shows that difference must be zero, i.e., $\mu = \nu$ and hence the
restriction of $\phi$ to the embedded $l^\infty(\lambda)$ is normal. We
conclude that $\phi$ is completely additive, as desired.

(iii) $\Rightarrow$ (ii) \ Let $\kappa$ be as in part (iii) and suppose
$M$ is not ${<}\,\kappa$-decomposable. Then there is a family
$\{p_\alpha: \alpha \in \kappa\}$ of mutually orthogonal nonzero projections
in $M$ with sum $1$. As in the proof of Theorem \ref{vnalg1},
let $N = \bigoplus p_\alpha M p_\alpha$, for each
$\alpha$ let $\psi_\alpha$ be a normal state on $p_\alpha M p_\alpha$
and use these states to build a normal map
$\omega: N \to l^\infty(\kappa)$. Let $E: M \to N$ be the canonical normal
conditional expectation, and by Theorem \ref{regular} let $\phi$ be a
regular singular state on $l^\infty(\kappa)$ given by integration against
a finitely additive regular probability measure $\mu$ which vanishes on
singletons. We claim that $\rho = \phi\circ \omega\circ E$ is a regular
singular state on $M$.

It is clear that $\rho$ is a state. For regularity, let $\{q_n\}$ be a
countable family of projections in ${\rm ker}\, \rho$. Then for each $n$,
the integral $\int \omega(E(q_n))\, d\mu$ is zero, and just as in the
proof of Theorem \ref{regular} this implies that
$\int \omega(E(\bigvee q_n))\, d\mu$ is zero, i.e., $\rho(\bigvee q_n) = 0$.
The argument at the end of the proof of Theorem \ref{vnalg1} shows that
$\rho$ is singular.
\end{proof}

As in Theorem \ref{chthm}, under the continuum hypothesis condition (iii)
can be simplified.

\begin{corollary}\label{chcor}
Assuming the continuum hypothesis, Ueda's theorem fails for a von Neumann
algebra $M$ if and only if $M$ contains a family of $\kappa$ many mutually
orthogonal nonzero projections where $\kappa$ is a measurable cardinal.
\end{corollary}

We may also note that if Ueda's theorem fails for $M$ then it also fails
for every nonselfadjoint subalgebra of $M$.

Theorem \ref{regularsingular} now yields information about the validity
of Ueda's theorem.

\begin{corollary}\label{uedacor}
Let $\kappa$ be an infinite cardinal.
\begin{itemize}
\item [(i)] If $\kappa = \aleph_1$ then Ueda's theorem holds for both
$l^\infty(\kappa)$ and $B(l^2(\kappa))$.
\item [(ii)] If $\kappa$ is Ulam real-valued measurable then Ueda's theorem
fails for both $l^\infty(\kappa)$ and $B(l^2(\kappa))$.
\end{itemize}
\end{corollary}

In particular, it is a theorem of ZFC that there are non $\sigma$-finite von
Neumann algebras which satisfy Ueda's theorem.

Regarding part (ii), we wish to emphasize that (subject to the consistency
of measurable cardinals) it is consistent that $2^{\aleph_0}$ is Ulam
real-valued measurable, and so it is consistent that Ueda's theorem can
already fail for $l^\infty(2^{\aleph}_0)$.

\end{document}